\documentclass{article}
\usepackage[T1]{fontenc}
\usepackage{authblk}
\usepackage{CJK}
\usepackage{amssymb,amsfonts,amsmath,amsthm}
\usepackage[dvips]{graphicx}
\usepackage{overpic}
\usepackage{subfigure}
\usepackage{color}
\usepackage{framed}
\definecolor {shadecolor}{rgb}{0.92,0.92,0.92}
\usepackage{listings}
\usepackage[colorlinks,citecolor=red,CJKbookmarks=true]{hyperref}
\usepackage{mathrsfs}
\usepackage[top=1in, bottom=1in, left=1.25in, right=1.25in]{geometry}

\newtheorem{ex}{Exercise}[section]
\newtheorem{thm}[ex]{Theorem}

\newtheorem{prop}[ex]{Proposition}

\newtheorem{rmk}[ex]{Remark}
\newtheorem{fac}[ex]{Fact}
	
\begin{document}
		\title{A solution to a question by A.Wigderson and Y.Wigderson}
	\author{Yiyu Tang}
	\date{Laboratoire d'analyse et de math\'ematiques appliqu\'ees, Universit\'e Gustave Eiffel}
\maketitle
\section{Introduction}
\paragraph{}In this note, we give a solution to a little conjecture proposed by A. Wigderson and Y. Wigderson when they consider a family of "Heisenberg-like" uncertainty principle in their paper \cite{refww}\\
\par 
\textbf{Question} (\cite{refww}, Conjecture 4.13).
	For $1<q\leq\infty$ and $q\neq2$, define $F_q:\mathscr{S}(\mathbb{R})\setminus\{0\}\rightarrow \mathbb{R}_{>0}$ by
	\begin{equation*}
	\begin{aligned}
	F_q(f):=\frac{||f||_q||\hat{f}||_q}{||f||_2||\hat{f}||_2}
	\end{aligned}
	\end{equation*}
Is $F_q$ surjective? i.e. is the image of $F_q$ all of $\mathbb{R}_{>0}$? Here, $\mathscr{S}(\mathbb{R})$ is the Schwartz space on real line, and $\hat{f}$ is the Fourier transform defined on $\mathscr{S}(\mathbb{R})$ by $\hat{f}(\xi)=\int_\mathbb{R}f(x)\mathrm{e}^{-2\pi\mathrm{i}x\xi}\mathrm{d}x$.
Finally, notation $||f||_q$ is the $L^q$ norm of $f$: $||f||_q=\big(\int_\mathbb{R}|f|^q\mathrm{d}x\big)^q$.
\par The classical Heisenberg uncertainty principle says $||f||_2||\hat{f}||_2\leq4\pi(\int|xf(x)|^2\mathrm{d}x)^{1/2}(\int|\xi \hat{f}(\xi)|^2\mathrm{d}\xi)^{1/2}$ when $f\in\mathscr{S}(\mathbb{R})$. If, for example, the image of  $F_q$ is bounded below by $c>0$, then it tells us $||f||_2||\hat{f}||_2\leq\frac{1}{c}||f||_q||\hat{f}||_q$, this could be considered as a variant of classical uncertainty principle.
\paragraph{}We answer this question affirmatively when $2<q<\infty$(the case $q=\infty$ has been proved in \cite{refww}), and negatively when $1<q<2$. The author got this result in May 2021, and presented it in a conference \textbf{ComPlane: the next generation} in June 17-18, 2021, see \cite{conf}. The same result was also obtained in 19th July, 2021 by three people, Linzhe Huang, Zhengwei Liu, Jinsong Wu, see \cite{refhlw}.
\begin{thm}
	If $2<q<\infty$, then the image of $F_q$ is all $\mathbb{R}_{>0}$. If $1<q<2$, then the image of $F_q$ is bounded below by $1$, i.e. $	F_q(f)\geq 1$ for any $f\in\mathscr{S}(\mathbb{R})\setminus\{0\}$.

\end{thm}
\paragraph{}Also, we can consider the general pair $(q,p)$ instead of $(q,2)$. This is the following theorem.

\begin{thm}\label{generalpq}Assuming $1<q<p<\infty$, and consider the mapping
	\begin{equation*}
	F_{q,p}(f)=\frac{||f||_q||\hat{f}||_q}{||f||_p||\hat{f}||_p},\ f\in\mathscr{S}(\mathbb{R})\setminus\{0\}.
	\end{equation*}
(i) If $\frac{1}{p}+\frac{1}{q}\geq1$, then image of $F_{q,p}$ is an infinite subinterval of $[1,\infty)$.\\
(ii) If $\frac{1}{p}+\frac{1}{q}<1$, then the image of $F_{q,p}$ is $(0,\infty)$.
\end{thm}

\begin{rmk}
	During the proof of the theorem, one can easily find that the bound $1$ is far from better. In fact, we use the Hausdorff-Young inequality several times in the proof, which says that $||\hat{f}||_{p^\prime}\leq||f||_p$ when $p\in[1,2]$. However, the sharp inequality should be $||\hat{f}||_{p^\prime}\leq\sqrt{(p^{1/p})/(p^\prime)^{1/p^\prime}}||f||_p$. This is a famous theorem due to W. Beckner. Maybe find the infimum of image of $F_q$ is more difficult than determine whether $F_q$ is surjective.
\end{rmk}

\section{Several facts}
\paragraph{} We recall the topology on $\mathscr{S}(\mathbb{R})$, for $\alpha,\ \beta\in\mathbb{N}$, define
$\rho_{\alpha,\beta}(f)=\sup_{x\in\mathbb{R}}\big|x^\alpha\frac{\mathrm{d}^\beta}{\mathrm{d}x^\beta}f(x)\big|$. Let $\{\rho_j\}_j$ be an enumeration of $\{\rho_{\alpha,\beta}\}_{\alpha,\beta\in\mathbb{N}}$, the following is a metric on $\mathscr{S}(\mathbb{R})$:
\begin{equation*}
d(f,g)=\sum_{j=1}^\infty2^{-j}\frac{\rho_j(f-g)}{1+\rho_j(f-g)}.
\end{equation*}
Under this metric, space $(\mathscr{S}(\mathbb{R}),d)$ is complete, so $\mathscr{S}(\mathbb{R})$ is a Frechet space. A sequence $\{f_n\}\subset\mathscr{S}(\mathbb{R})$ is said to converges to an $f\in\mathscr{S}(\mathbb{R})$ in $\mathscr{S}(\mathbb{R})$, if $d(f_n,f)$ goes to $0$ when $n$ tends to infinity.

\par The following facts are basic properties of $\mathscr{S}(\mathbb{R})$ and Fourier transform on $\mathscr{S}(\mathbb{R})$. All of them can be found in any textbook of Fourier analysis.
\begin{fac}
Firstly, the Fourier transform $f\mapsto\hat{f}$ is a homeomorphism on $\mathscr{S}$ to itself.
Secondly, suppose that functions $\{f_n\}_{n\geq1}$ and $f$ belong to $\mathscr{S}(\mathbb{R})$, if $f_n$ converges to $f$ in $\mathscr{S}(\mathbb{R})$, then $f_n$ converges to $f$ in $L^p$ for all $1\leq p\leq\infty$. Therefore, the function $f\mapsto||f||_p$ is continuous on $\mathscr{S}(\mathbb{R})$ for all $1\leq p\leq \infty$.
\end{fac}
\begin{proof}
	For a proof, see \cite{cfa}, page 106, Proposition 2.2.6 and page 113, Corollary 2.2.15.
\end{proof}

Combining these facts, the mapping $f\mapsto||\hat{f}||_p$ is continuous as a composition of two continuous mappings, so the mapping $f\mapsto F_q(f)$ is continuous from $\mathscr{S}(\mathbb{R})\setminus\{0\}$ to $\mathbb{R}_{>0}$. As a metric space, $\mathscr{S}(\mathbb{R})\setminus\{0\}$ is connected, since $\mathscr{S}(\mathbb{R})$ is a Frechet space with $\dim\mathscr{S}(\mathbb{R})=\infty$. Recall that between two metric spaces, continuous image of a connected set is also connected, and connected sets in $\mathbb{R}$ are intervals, we conclude that
\begin{prop}
	The image of $F_q$ is an interval on $\mathbb{R}_{>0}$.
\end{prop}	
\par Now we only need to study the endpoint of $Im(F_q)$. But firstly, let's see what could we get by using some simple functions to calculate $F_q$. We use the same construction in \cite{refww}.
\paragraph{}For $a>1$, define functions $f_{a}$ by
\begin{equation*}
f_{a}(x)=\mathrm{e}^{-\pi(a^2-1)x^2}\mathrm{e}^{-2\pi\mathrm{i}ax^2}.
\end{equation*}
Clearly, functions $f_{a}\in\mathscr{S}(\mathbb{R})$, and one can calculate the Fourier transform of $f_{a}$:
\begin{equation*}
\hat{f}_{a}(\xi)=\frac{1}{a+\mathrm{i}}\exp\bigg(\frac{-\pi\xi^2(a^2-1)}{(a^2+1)^2}\bigg)\exp\bigg(-2\pi\mathrm{i}\xi^2\frac{a}{(a^2+1)^2}\bigg).
\end{equation*}
Since $\int_\mathbb{R}\mathrm{e}^{-\pi x^2}\mathrm{d}x=1$, we get, for any $1\leq q<\infty$,
\begin{equation*}
\big|\big|f_{a}\big|\big|_{L^q}=\bigg(\frac{1}{\sqrt{q(a^2-1)}}\bigg)^{\frac{1}{q}},\ \big|\big|\hat{f}_{a}\big|\big|_{L^q}=\frac{1}{\sqrt{a^2+1}}\bigg(\frac{a^2+1}{\sqrt{q(a^2-1)}}\bigg)^{\frac{1}{q}}.
\end{equation*}
Therefore, one calculates $F_q$ precisely:
\begin{equation*}
F_q(f_{a})=\frac{||f_{a}||_q||\hat{f}_{a}||_q}{||f_{a}||^2}=\sqrt{2}\ \bigg(\frac{1}{q}\bigg)^{1/q}\bigg(\frac{a^2+1}{a^2-1}\bigg)^{1/q-1/2}=\sqrt{2}\ \bigg(\frac{1}{q}\bigg)^{1/q}\bigg(\frac{t+1}{t-1}\bigg)^{1/q-1/2},\ a^2=t>1.
\end{equation*}
We notice that the mapping $t\mapsto\frac{t+1}{t-1}$ is decreasing when $t>1$, so if $1/q-1/2>0$, the image of the mapping $t\mapsto(\frac{t+1}{t-1})^{1/q-1/2},\ t>1$
is $(1,\infty)$. Similarly, if $1/q-1/2<0$, then image of the above mapping is $(0,1)$. In conclusion, by testing functions $f_{a}$ on $F_q$, we have:
\begin{prop}\label{basic bound estimate}
	(i) If $1<q<2$, then at least $\big(\sqrt{2}(\frac{1}{q})^{\frac{1}{q}},\infty\big)\subset Im(F_q)$.\\
(ii) If $2<q<\infty$, then at least $\big(0,\sqrt{2}(\frac{1}{q})^{\frac{1}{q}}\big)\subset Im(F_q)$.
\end{prop}

\section{The case $2<q<\infty$}
\paragraph{}In this section, we always assume that $2<q<\infty$, and we will prove that $Im(F_q)=(0,\infty)$. By Proposition \ref{basic bound estimate}, we only need to construct a sequence $\{f_n\}_n\subset\mathscr{S}(\mathbb{R})$, so that $F_q(f_n)\rightarrow\infty$ as $n$ goes to infinity. We use the same construction in \cite{refww}.
\paragraph{} For each $c>0$, define the function
\begin{equation*}
g_c(x)=\frac{1}{\sqrt{c}}\ \mathrm{e}^{-\pi\frac{x^2}{c^2}}+\sqrt{c}\ \mathrm{e}^{-\pi c^2x^2}.
\end{equation*}
Clearly, functions $g_c\in\mathscr{S}(\mathbb{R})$ and $\hat{g_c}=g_c$ for all $c>0$, so $F_q(g_c)=\big(\frac{||g_c||_q}{||g_c||_2}\big)^2$. A direct calculation shows that 
\begin{equation}
||g_c||^2_2=\sqrt{2}+\frac{2c}{\sqrt{c^4+1}},\ c>0.\label{L2norm}
\end{equation}
Now we estimate $||g_c||^2_q$, by definition
\begin{equation*}
\begin{aligned}
||g_c||^2_q=\bigg\{\int\bigg(\frac{1}{c}\ \mathrm{e}^{-2\pi\frac{x^2}{c^2}}+c\ \mathrm{e}^{-2\pi c^2x^2}+2\mathrm{e}^{-\pi(c^2+1/c^2)x^2}\bigg)^{\frac{q}{2}}\mathrm{d}x\bigg\}^{\frac{2}{q}}.\\
\end{aligned}
\end{equation*}
Notice that $\big(\sum_i a_i\big)^{q/2}\geq\sum_ia_i^{q/2}$ when $q/2>1$, therefore
\begin{equation}\label{keyineq}
\begin{aligned}
||g_c||^2_q\geq&\bigg\{\int\bigg(\frac{1}{c}\bigg)^{q/2}\ \mathrm{e}^{-\pi q\frac{x^2}{c^2}}+c^{q/2}\ \mathrm{e}^{-\pi q c^2x^2}+2^{q/2}\mathrm{e}^{\frac{-\pi q}{2}(c^2+1/c^2)x^2}\mathrm{d}x\bigg\}^{\frac{2}{q}}\\
=&\bigg\{\bigg(\frac{1}{c}\bigg)^{q/2}\frac{1}{\sqrt{q/c^2}}+ c^{q/2}\frac{1}{\sqrt{qc^2}}+2^{q/2}\frac{1}{\sqrt{q(c^2+1/c^2)/2}}\bigg\}^{\frac{2}{q}}.
\end{aligned}
\end{equation}
In conclusion, for all $c>0$:
\begin{equation}
||g_c||^2_q\geq\bigg\{\frac{1}{\sqrt{q}}c^{1-\frac{q}{2}}+ \frac{1}{\sqrt{q}}c^{\frac{q}{2}-1}+\frac{2^{(q+1)/2}}{\sqrt{q}}\frac{c}{\sqrt{c^4+1}}\bigg\}^{\frac{2}{q}}\geq \bigg(\frac{1}{\sqrt{q}}c^{\frac{q}{2}-1}\bigg)^{2/q}=\bigg(\frac{1}{q}\bigg)^{1/q}c^{1-\frac{2}{q}}.\label{Lqnorm}
\end{equation}
\paragraph{}Combining \eqref{L2norm} and \eqref{Lqnorm}, we get
\begin{equation*}
F_q(g_c)\geq\bigg(\frac{1}{q}\bigg)^{1/q}\frac{c^{1-\frac{2}{q}}}{\sqrt{2}+\frac{2c}{\sqrt{c^4+1}}}.
\end{equation*}
Finally, let $c\rightarrow\infty$, and we have $F_q(g_c)\rightarrow\infty$.

\section{The case $1<q<2$}\label{q_between_1_and_2}
\paragraph{}In this section, we always assume that $1<q<2$. In fact, the case $q<2$ is easier. Let $q^\prime$ be the conjugate exponent of $q$. H\"older inequality implies
\begin{equation}
||f||_2\leq ||f||^{1/2}_q||f||^{1/2}_{q^\prime}.
\end{equation}
\par Finally, recall the Hausdorff-Young inequality, when $1\leq q\leq 2$, we have
\begin{equation*}
||\hat{f}||_{q\prime}\leq||f||_q,\ f\in\mathscr{S}.
\end{equation*}
Therefore, we choose $g=\hat{f}\in\mathscr{S}$ and we have $||\hat{g}||_{q^\prime}\leq||g||_q$, i.e. 
\begin{equation}
||f||^{1/2}_{q^\prime}\leq||\hat{f}||^{1/2}_q.
\end{equation}
\par Combining these inequalities, we get
\begin{equation*}
||f||_2\leq||f||^{1/2}_q||\hat{f}||^{1/2}_q,
\end{equation*}
which is to say $\frac{||f||_q||\hat{f}||_q}{||f||^2_2}\geq1$ for any $f\in\mathscr{S}(\mathbb{R})\setminus\{0\}$.

\section{General $p,q$}
In this section, we discuss a general case 
\begin{equation*}
F_{q,p}(f)=\frac{||f||_q||\hat{f}||_q}{||f||_p||\hat{f}||_p},\ 1<q<p<\infty.
\end{equation*}
\par Firstly, let's see what could we get by using special functions. We still use the functions $f_{a}$ in Section 2. Recall
\begin{equation*}
\begin{aligned}
f_{a}(x)&=\mathrm{e}^{-\pi(a^2-1)x^2}\mathrm{e}^{-2\pi\mathrm{i}ax^2},\ a>1.\\
\end{aligned}
\end{equation*}
A direct calculation shows that for each $1<q<\infty$,
\begin{equation*}
\begin{aligned}
||f_a||_q=\bigg(\frac{1}{q(a^2-1)}\bigg)^\frac{1}{2q},\ ||\hat{f}_a||_q=(a^2+1)^{-1/2}\bigg(\frac{a^2+1}{\sqrt{q(a^2-1)}}\bigg)^{\frac{1}{q}}.
\end{aligned}
\end{equation*}
\paragraph{}
Replacing $q$ by $p$, we have
\begin{equation*}
F_{q,p}(f_a)=\frac{(1/q)^{1/q}}{(1/p)^{1/p}}\bigg(\frac{a^2+1}{a^2-1}\bigg)^{\frac{1}{q}-\frac{1}{p}}=\frac{(1/q)^{1/q}}{(1/p)^{1/p}}\bigg(\frac{t+1}{t-1}\bigg)^{\frac{1}{q}-\frac{1}{p}},\ t=a^2>1.
\end{equation*}
So at least $\bigg(\frac{(1/q)^{1/q}}{(1/p)^{1/p}},\infty\bigg)\subset Im(F_{q,p})$.
\subsection{Theorem \ref{generalpq}, case (i)}
\paragraph{}Now we prove (i) of Theorem \ref{generalpq}: $\frac{1}{p}+\frac{1}{q}\geq1$.  Note that $q<2$ always holds in this case. Dividing it into two sub-cases: $q<p<2$ and $q<2\leq p$. Two sub-cases are a little bit different, although results are the same.
\subsubsection{When $q<p<2$}
\paragraph{}Writing $p$ as a convex combination of $q$ and $2$:
\begin{equation*}
p=\lambda q+(1-\lambda)2,\ 0<\lambda<1.
\end{equation*}
Then $\int|f|^p=\int|f|^{\lambda q}|f|^{(1-\lambda)2}$. Notice that $\frac{1}{\lambda}\in(1,\infty)$, by H\"older's inequality ($\lambda+(1-\lambda)=1$), we have
\begin{equation}\label{interpolation ineq}
\begin{aligned}
||f||_p\leq||f||^{\frac{1/p-1/2}{1/q-1/2}}_q||f||^{\frac{1/q-1/p}{1/q-1/2}}_2.
\end{aligned}
\end{equation}
Inequality \eqref{interpolation ineq} also holds for $\hat{f}$,
\begin{equation}\label{interpolation ineq 2}
\begin{aligned}
||\hat{f}||_p\leq||\hat{f}||^{\frac{1/p-1/2}{1/q-1/2}}_q||\hat{f}||^{\frac{1/q-1/p}{1/q-1/2}}_2.
\end{aligned}
\end{equation}
\par Combining the above two inequalities and definition of $F_{q,p}$, we have
\begin{equation*}
F_{q,p}(f)\geq\Bigg(\frac{||f||_q||\hat{f}||_q}{||f||_2||\hat{f}||_2}\Bigg)^{\frac{1/q-1/p}{1/q-1/2}}=\big(F_q(f)\big)^{\frac{1/q-1/p}{1/q-1/2}}.
\end{equation*}
Finally, notice that $\frac{1/q-1/p}{1/q-1/2}>0$, so if $1<q<p<2$, then we come back to the case $1<q<2,\ p=2$ in Section \ref{q_between_1_and_2}, which asserts that image of $F_q$ is bounded below by $1$. Therefore, the image of $F_{q,p}$ is bounded below by $1$.

\subsubsection{When $q<2\leq p$}
\paragraph{}Since $1<p^\prime\leq2$, by Hausdorff-Young inequality $||\hat{f}||_p\leq||f||_{p^\prime}.$
This shows that 
\begin{equation*}
\frac{||f||_q||\hat{f}||_q}{||f||_p||\hat{f}||_p}\geq\frac{||f||_q||\hat{f}||_q}{||\hat{f}||_{p^\prime}||f||_{p^\prime}}=F_{q,p^\prime}(f).
\end{equation*}
Now we come back to the case in subsection 5.1.1, because $1< q\leq p^\prime<2$.

\subsection{Theorem \ref{generalpq}, case (ii)}
\paragraph{}Now we discuss the case $\frac{1}{p}+\frac{1}{q}<1$. Also, remember that $q<p$, therefore $p>2$. Here we use the same function $g_c$ in section 3. Recall that for $c>0$,
\begin{equation*}
g_c(x)=\frac{1}{\sqrt{c}}\ \mathrm{e}^{-\pi\frac{x^2}{c^2}}+\sqrt{c}\ \mathrm{e}^{-\pi c^2x^2},\ g_c=\hat{g_c}.
\end{equation*}
The inequality \eqref{keyineq} shows that ($p>2$)
\begin{equation*}
||g_c||^2_p\geq\bigg\{\frac{1}{\sqrt{p}}c^{1-\frac{p}{2}}+ \frac{1}{\sqrt{p}}c^{\frac{p}{2}-1}+\frac{2^{(p+1)/2}}{\sqrt{p}}\frac{c}{\sqrt{c^4+1}}\bigg\}^{\frac{2}{p}}\thicksim \bigg(\frac{c^{\frac{p}{2}-1}}{\sqrt{p}}\bigg)^{2/p},\ \text{as\ }c\rightarrow\infty.
\end{equation*}
Similarly, for $q>1$, we have $(a_1+a_2+a_3)^{\frac{q}{2}}\leq \max\{1,3^{q/2-1}\}\bigg( a_1^{\frac{q}{2}}+ a_2^{\frac{q}{2}}+ a_3^{\frac{q}{2}}\bigg)$, so we just reverse the inequality \eqref{keyineq} and get
\begin{equation*}
\begin{aligned}
||g_c||^2_q\leq&\frac{1}{\sqrt{2}}+ \frac{1}{\sqrt{2}}+\frac{2}{\sqrt{2}}\frac{c}{\sqrt{c^4+1}}\leq4,\ \text{when\ }q=2.\\
||g_c||^2_q\leq&\bigg\{\frac{1}{\sqrt{q}}c^{1-\frac{q}{2}}+ \frac{1}{\sqrt{q}}c^{\frac{q}{2}-1}+\frac{2^{(q+1)/2}}{\sqrt{q}}\frac{c}{\sqrt{c^4+1}}\bigg\}^{\frac{2}{q}}\thicksim \bigg(\frac{c^{1-\frac{q}{2}}}{\sqrt{q}}\bigg)^{2/q},\ \text{as\ }c\rightarrow\infty,\ \text{when\ }q<2.\\
||g_c||^2_q\leq&3^{1-\frac{2}{q}}\bigg\{\frac{1}{\sqrt{q}}c^{1-\frac{q}{2}}+ \frac{1}{\sqrt{q}}c^{\frac{q}{2}-1}+\frac{2^{(q+1)/2}}{\sqrt{q}}\frac{c}{\sqrt{c^4+1}}\bigg\}^{\frac{2}{q}}\thicksim 3^{1-\frac{2}{q}}\bigg(\frac{c^{\frac{q}{2}-1}}{\sqrt{q}}\bigg)^{2/q},\ \text{as\ }c\rightarrow\infty\ \text{when\ }q>2.
\end{aligned}
\end{equation*}
Consequently, as $c$ goes to infinity, we have $||g_c||^2_q/||g_c||^2_p\leq C_{p,q}\ c^{2(1/q+1/p-1)}$ when $q\leq2$ and $||g_c||^2_q/||g_c||^2_p\leq C_{p,q}\ c^{2(1/p-1/q)}$ when $q>2$, where $C_{p,q}$ depends only on $p$ and $q$. Finally, just note that $1/q+1/p-1<0$ and $1/p-1/q<0$, so $\lim_{c\rightarrow\infty}||g_c||^2_q/||g_c||^2_p=0$.

\end{document}